\documentclass[10pt]{amsart}
\usepackage[cp1251]{inputenc}
\usepackage[english,russian]{babel}
\usepackage{amsmath}
\usepackage{amssymb}
\usepackage{amsfonts}

\newtheorem{theorem}{Теорема}
\newtheorem{definition}{Определение}

\begin{document}
\renewcommand{\refname}{References}

\thispagestyle{empty}

\title[On an equality for the iterated  spherical mean ]{On an equality for the iterated weighted spherical mean and its applications}
\author{{Shishkina E.L., Sitnik S.M.}}%

\address{Elina Leonidovna Shishkina  
\newline\hphantom{iii} Voronezh State University,
\newline\hphantom{iii} Universitetskaya pl. 1, Voronezh, 394000, Russia.}
\email{ilina\_dico@mail.ru}%

\address{Sergei Mihailovich Sitnik  
\newline\hphantom{iii} Voronezh Institute of the Ministry of Internal Affairs,
\newline\hphantom{iii} pr. Patriotov, 53,  Voronezh, 394065, Russia;
\newline\hphantom{iii} \textit{and}
\newline\hphantom{iii} Peoples' Friendship University of Russia,
\newline\hphantom{iii} M. - Maklaya str., 6, Moscow, 117198, Russia.}%
\email{mathsms@yandex.ru}%

\thanks{\sc Shishkina, E.L., Sitnik, S.M.
On an identity for the iterated weighted spherical mean and its applications.}

\vspace{1cm}
\maketitle {\small
\begin{quote}
\noindent{\sc Abstract. } Spherical means are well--known useful tool in the theory of partial differential equations with applications to solving hyperbolic and ultrahyperbolic equations and problems of integral geometry, tomography and Radon transforms. We generalize iterated spherical means to weighted ones based on generalized translation operators and consider applications to $B$--hyperbolic  equations and transmission tomography problems.

\medskip

 \end{quote}
}



\section{Введение}

Сферические средние имеют многочисленные применения в теоретической математике и её приложениях. В классических трудах \cite{1}--\cite{3} приводятся различные применения сферических средних к теории уравнений в частных производных, в том числе эллиптического, гиперболического и ультрагиперболического типов. Кроме того, сферические средние являются объектом изучения интегральной геометрии с приложением результатов исследований к томографии \cite{Nat}, например фотоакустической \cite{Gorner}. Известна также тесная связь сферических средних с преобразованием Фурье и потенциалом М.~Рисса (см. \cite{Helgason2}). Представление решений различных дифференциальных уравнений с частными производными с помощью сферических средних также связано с теорией операторов преобразования \cite{Sita1}--\cite{Sita2}. При этом существует известный подход к теории преобразования Радона, основанный на разложении функций в ряды по сферическим гармоникам \cite{Dea}, восходящий к основополагающей работе Д.~Людвига (см. \cite{Nat}, \cite{Helgason2}). При этом подходе фундаментальную роль играют операторы преобразования Бушмана--Эрдейи (или Гегенбауэра--Чебышёва, как принято называть их в работах по томографии), см.
\cite{Sita3}--\cite{Rub2}.

В этой работе мы будем рассматривать \emph{весовое сферическое среднее}, в котором по каждой переменной подынтегральной функции действует обобщённый сдвиг, определенный  в \cite{Delsartes} (см. также \cite{levitan}). Такое среднее изучалось в работах \cite{LPSh1}, \cite{LPSh2}. В \cite{LPSh1} были получены свойства весового сферического среднего и была доказана теорема о среднем значении типа теоремы Асгейрссона (см. \cite{Asgeirsson}).  Весовое сферическое среднее в \cite{LPSh2}  было использовано для построения решения  задачи Коши для гиперболического уравнения, содержащего сингулярные дифференциальные операторы Бесселя по каждой из переменных (см. также \cite{Fox}, где решение этой задачи построено другим методом).

В настоящей работе доказаны новые тождества для итерированных весовых сферических средних, которые необходимы для получения явных формул при восстановлении функции по её весовым сферическим средним. Кроме того, с использованием весовых сферических  средних в работе  получена формула для представления функции через операторы обобщённого сдвига и преобразование Фурье--Бесселя или Ханкеля.
 Такие формулы используются в различных прикладных задачах томографии и интегральной геометрии.

\section{Основные определения}

Будем рассматривать функции $f(x)$, определённые в области
$$
\mathbb{R}^+_n{=}\{x{=}(x_1,\ldots,x_n)\in\mathbb{R}_n,\,\,\, x_1{>}0,\ldots, x_n{>}0\}.
$$

\begin{definition}
Следуя \cite{singul} (см. стр. 21),  функцию $f\in C^\infty(\mathbb{R}^+_n)$ будем называть \emph{чётной} по $x_i$, если
$\frac{d^{2k+1}f}{d x^{2k+1}_i}\biggr|_{x_i=0}=0$ при всех $k=0,1,2...$. Если $f(x)$ при $x\in\mathbb{R}^+_n$ не является бесконечно дифференцируемой, то будем называть эту функцию \emph{четной} по $x_i$, если её можно продолжить чётным образом на значения $x_i\leq 0$ с сохранением её свойств.
\end{definition}

Пространство $S_{ev}(\mathbb{R}_n^+)=S_{ev} -$\, это подпространство
 пространства Шварца \, $S(\mathbb{R}_n^+)$\,, состоящее из функций\,
$\varphi (x)$\,, чётных по каждой из своих переменных.

Пусть $\gamma=(\gamma_1,\ldots,\gamma_n)$ --- мультииндекс, состоящий из фиксированных положительных чисел, и $|\gamma|=\gamma_1+\ldots+\gamma_n$ его длина.

  Положим $\mathbf{j}_\gamma(x,\xi)=\prod\limits_{i=1}^nj_{\frac{\gamma_i-1}{2}}(x_i\,\xi_i),$ $|\xi|=1$, где $j_{\frac{\gamma_i-1}{2}}$  --- "малая"\  или "нормированная"\  функция Бесселя, которая  связана с "большой"\  функцией
Бесселя первого рода $J_{\gamma_i-1\over2}$ равенством
\begin{equation}\label{BessFunk}
j_\omega(t){=}\Gamma(\omega+1)\,\left(\frac{2}{t}\right)^{\omega}J_\omega(t),\quad\omega>0, \quad t\in\mathbb{R}_1.
\end{equation}

Прямое  и обратное преобразование Фурье--Бесселя для функции $f{\in} S_{ev}$ определяются по формулам
 (см. \cite{singul}):
$$
F_B[f](\xi)
=\widehat{f}(\xi) = \int_{\mathbb{R}_n^+}
\mathbf{j}_{\gamma}(x,\xi)\,f(x)\,x^\gamma\,\,dx\,,\qquad x^\gamma=\prod\limits_{i=1}^n x_i^{\gamma_i},
$$
$$
f(x)=F_B^{-1}[\widehat{f}(\xi)](x) =
\frac{2^{n-|\gamma|}}{\prod\limits_{j=1}^n\Gamma^2
\left(\gamma_j+1\over2\right)}\,F_B[\widehat{f}(\xi)](x).
$$

Часть шара  $|x|\leq r$, $|x|=\sqrt{x_1^2+...+x_n^2}$, принадлежащую $\mathbb{R}^+_n$, будем обозначать
$B_r^+(n)$. Граница $B_r^+(n)$ состоит из части сферы $S_r^+(n){=}\{x\in\mathbb{R}^+_n:|x|{=}r\}$ и из частей координатных гиперплоскостей $x_i{=}0$, $i{=}1,{\ldots},n$, таких что $|x|{\leq}r$\,.

Рассмотрим функцию $g(s)$ одной переменной $s>0$, чётную в смысле определения 1,  и пусть $\nu>0$.
В  \cite{Delsartes} определён, а в  \cite{levitan} (см. стр. 121 и далее) подробно изучен \emph{оператор обобщённого сдвига} вида
\begin{equation}\label{sdvig1}
^\nu \mathcal{T}^t_s g(s)=\frac{\Gamma\left(\frac{\nu+1}{2}\right)}{\sqrt{\pi}\Gamma\left(\frac{\nu}{2}\right)}\int\limits_0^\pi
g(\sqrt{s^2+t^2-2st\cos{\varphi}})\sin^{\nu-1}{\varphi}d\varphi.
\end{equation}
Если рассматриваемая функция $g(s)\in C^2(\mathbb{R}_1^+)$, то $\Phi(s,t)=\,^\nu \mathcal{T}^t_s g(s)$ представляет собой единственное решение задачи Коши вида
$$
\frac{\partial^2 \Phi}{\partial s^2}+\frac{\nu}{s}\frac{\partial \Phi}{\partial s}=\frac{\partial^2 \Phi}{\partial t^2}+\frac{\nu}{t}\frac{\partial \Phi}{\partial t},
$$
$$
\Phi(s,0)=g(s),\qquad \frac{\partial \Phi}{\partial t}\biggr|_{t=0}=0.
$$
Оператор \eqref{sdvig1} заменой переменных сводится к виду
\begin{equation}\label{sdvig2}
^\nu \mathcal{T}^t_s g(s)=\frac{2}{(2st)^{\gamma-1}}\frac{\Gamma\left(\frac{\nu+1}{2}\right)}{\sqrt{\pi}\Gamma\left(\frac{\nu}{2}\right)}\int\limits_{|x-y|}^{x+y}zf(z)[(z^2-(x-y)^2)((x+y)^2-z^2)]^{\frac{\gamma}{2}-1}dz.
\end{equation}

Далее будем считать, что $x\in\mathbb{R}_n^+$, $f(x)$  --- интегрируемая с весом $x^\gamma=\prod\limits_{i=1}^n x_i^{\gamma_i}$ функция, чётная по каждой координате своего аргумента в смысле определения 1.   Рассмотрим многомерный обобщённый сдвиг
$$T^y f(x) =\prod_{i=1}^n (^{\gamma_i}\mathcal{T}_{x_i}^{y_i}
f)(x),
$$
где каждый из одномерных обобщенных сдвигов $^{\gamma_i}\mathcal{T}_{x_i}^{y_i}$ определен
формулами \eqref{sdvig1} или \eqref{sdvig2}.

\emph{Весовое сферическое среднее}, порождённое обобщенным сдвигом, имеет вид
\begin{equation}\label{EQ4}
M_r^\gamma f(x)=M^\gamma_f(x;r)=\frac{1}{|S_1^+(n)|_\gamma}\int\limits_{S^+_1(n)}T^{ry}f(x)y^\gamma
dS(y),\qquad  r>0,
\end{equation}
где коэффициент $|S^+_1(n)|_\gamma$  вычисляется по формуле
\begin{equation}\label{PlSh}
|S_1^+(n)|_\gamma
=\int\limits_{S^+_1(n)}\prod\limits_{i=1}^n x_i^{\gamma_i}dS(y)=
\frac{\prod\limits_{i=1}^n{\Gamma\left(\frac{\gamma_i{+}1}{2}\right)}}{2^{n-1}\Gamma\left(\frac{n{+}|\gamma|}{2}\right)}
\end{equation}
(см. \cite{LyahovKniga}, стр. 20, формула (1.2.5), где надо положить $N{=}n$).

Средние \eqref{EQ4}  изучались в  \cite{LPSh1}, \cite{LPSh2}, \cite{KiprZas}. В \cite{Lyah13} применялось весовое сферическое среднее, порожденное смешанным обобщённым сдвигом, который действовал только по одной переменной.

Многомерный оператор Пуассона определяется формулой (см. \cite{levitan}, \cite{singul})
\begin{equation}\label{Poison}
\mathcal{P}^\gamma_{x}f(x)=C(\gamma)\int\limits_0^\pi...\int\limits_0^\pi
f(x_1\cos\alpha_1,...,x_n\cos\alpha_n)
\prod\limits_{j=1}^n\sin^{\gamma_j-1}\alpha_j\,d\alpha_1{...}d\alpha_n,
\end{equation}
где
$$
C(\gamma)=\frac{\prod\limits^n_{i=1}\Gamma\left(\frac{\gamma_i+1}{2}\right)}{\sqrt{\pi}2^{n-1}\Gamma\left(\frac{|\gamma|+n-1}{2}\right)}.
$$

 Отметим, что для  интегрируемой по $B_R^+(n)$ с весом $x^\gamma$ функции $f(x)$ и для непрерывной при $t\in[0,\infty)$  функции одной переменной $g(t)$ посредством сферической замены переменных $x=\rho\theta$, $|x|=\rho$ получается  формула
 \begin{equation}\label{Ball}
\int\limits_{B_R^+(n)}g(|x|)f(x)\,x^\gamma\, dx=\int\limits_0^R g(\rho)\rho^{n+|\gamma|-1}d\rho \int\limits_{S_1^+(n)}f(\rho x)x^\gamma dS(x).
 \end{equation}
 Если в \eqref{Ball} положить $g(t)=1$ и продифференцировать обе части по $R$, то будем иметь
\begin{equation}\label{Ball1}
\int\limits_{S^+_1(n)}f(R x)x^\gamma dS(x)=R^{1-n-|\gamma|}\frac{d}{dR}\int\limits_{B_R^+(n)}f(x)x^\gamma dx.
 \end{equation}

Нам также потребуется  формула
\begin{equation}\label{VesPlVol}
\int\limits_{S_1^+(n)}\mathcal{P}_{\xi}^\gamma
g(\langle\xi, x\rangle)x^\gamma dS(x)
=C(\gamma)\int\limits_{-1}^1
g(|\xi|p)(1-p^2)^{\frac{n+|\gamma|-3}{2}}\:dp.
\end{equation}
 В книге \cite{LyahovKniga} (стр. 44, формула 1.7.13) приведена общая формула для случая, когда только часть переменных нагружены весами, из которой получается \eqref{VesPlVol} при отсутствии невесовых переменных. Отметим также, что общая формула, аналогичная формуле 1.7.13 из \cite{LyahovKniga}, получена  И.А. Киприяновым  и Л.А. Ивановым  в статье \cite{KiprIvanov} (формула (3) на стр. 57), однако в формуле (3) из \cite{KiprIvanov} отсутствует умножение на константу $C(\gamma)$.

\section{Итерированное весовое сферическое среднее и его свойства}

\emph{Итерированное весовое сферическое среднее} имеет вид (см. \cite{ShishGer}):
$$
I^{\,\gamma}_f(x;\lambda,\mu)=I^{\,\gamma}_{\lambda,\mu}f(x)=M^\gamma_\lambda M^\gamma_\mu f(x)=$$
$$
=\frac{1}{|S_1^+(n)|_\gamma^2}\int\limits_{S_1^+(n)}\int\limits_{S_1^+(n)}T_x^{\lambda\zeta}T_x^{\mu\xi}[f(x)]\zeta^\gamma
\xi^\gamma dS(\xi) dS(\zeta).
$$

Используя перестановочное свойство  обобщённого сдвига (см. \cite{levitan}, формула (7.1)), получим, что итерированное весовое сферическое среднее симметрично относительно $\lambda$ и $\mu$:
$$
I^{\,\gamma}_f(x;\lambda,\mu)=I^{\,\gamma}_f(x;\mu,\lambda).
$$

Очевидно, что
$$
I^{\,\gamma}_f(x;\lambda,0)=I^{\,\gamma}_f(x;0,\lambda)=M^\gamma_\lambda f(x)
$$
и
$$
I^{\,\gamma}_f(x;0,0)=f(x).
$$

 Следуя \cite{1} (см. стр. 73), докажем равенство, выражающее итерированное сферическое среднее $I^{\,\gamma}_f(x;\lambda,\mu)$ через однократный интеграл от весового сферического среднего $M^\gamma_f(x;r)$.

\begin{theorem}
 Пусть $f$ --- интегрируемая с весом $x^\gamma=\prod\limits_{i=1}^nx_i^{\gamma_i}$ функция, чётная по каждой из своих переменных в смысле определения 1. Тогда справедливы следующие формулы:
\begin{equation}\label{Transl}
I^{\,\gamma}_f(x;\lambda,\mu)=\,^\nu\mathcal{T}_\mu^\lambda M^\gamma_f(x;\mu),
\end{equation}
$$
I^{\,\gamma}_f(x;\lambda,\mu)=\frac{2\,\Gamma\left(\frac{n{+}|\gamma|}{2}\right)}
{\sqrt{\pi}\Gamma\left(\frac{|\gamma|+n-1}{2}\right)}\frac{1}{(2\lambda\mu)^{{n+|\gamma|-2}}}\times
$$
\begin{equation}\label{IterFor}
\times\int\limits_{\lambda-\mu}^{\lambda+\mu}
    \Bigl({\left(\lambda^2-(r-\mu)^2\right)\left((r+\mu)^2-\lambda^2\right)}
    \Bigr)^{\frac{n+|\gamma|-3}{2}}\,M^\gamma_f(x;r)\,rdr,
\end{equation}
$$
I^{\,\gamma}_f\left(x;\frac{\beta-\alpha}{2},\frac{\beta+\alpha}{2}\right)=
$$
\begin{equation}\label{05}
=\frac{\Gamma\left(\frac{n{+}|\gamma|}{2}\right)}{\sqrt{\pi}\Gamma\left(\frac{|\gamma|+n-1}{2}\right)}
\frac{2^{{n+|\gamma|-1}}}{(\beta^2-\alpha^2)^{{n+|\gamma|-2}}}\int\limits_{\alpha}^{\beta}
    \Bigl({(\beta^2-r^2)(r^2-\alpha^2)}\Bigr)^{\frac{n+|\gamma|-3}{2}}M^\gamma_f(x;r)rdr.
\end{equation}
\end{theorem}

\begin{proof}
Пусть $g(s)$ --- произвольная непрерывная финитная функция одной переменной. Рассмотрим интеграл
$$
J=\int\limits_0^{+\infty}\lambda^{n+|\gamma|-1}
g(\lambda)I^{\,\gamma}_f(x;\lambda,\mu)d\lambda=
$$
$$
=\frac{1}{|S_1^+(n)|_\gamma^2}\int\limits_0^{+\infty}\lambda^{n+|\gamma|-1}
g(\lambda) d\lambda\int\limits_{S_1^+(n)}\int\limits_{S_1^+(n)}T_x^{\lambda\zeta}T_x^{\mu\xi}[f(x)]\zeta^\gamma
\xi^\gamma dS(\xi) dS(\zeta).
$$
Применяя перестановочное свойство  обобщённого сдвига (7.1) из \cite{levitan} и формулу \eqref{Ball}, запишем это выражение в виде
$$
J=\frac{1}{|S_1^+(n)|_\gamma^2}\int\limits_{S_1^+(n)}T_x^{\mu\xi}\left[\int\limits_0^{+\infty}\lambda^{n+|\gamma|-1}
g(\lambda)d\lambda\int\limits_{S_1^+(n)}T_x^{\lambda\zeta}[f(x)]\zeta^\gamma
 dS(\zeta)\right]\xi^\gamma dS(\xi)=
$$
$$
=\frac{1}{|S_1^+(n)|_\gamma^2}\int\limits_{S_1^+(n)}T_x^{\mu\xi}\left[\lim\limits_{R\rightarrow +\infty}\int\limits_0^{R}\lambda^{n+|\gamma|-1}
g(\lambda)d\lambda\int\limits_{S_1^+(n)}T_x^{\lambda\zeta}[f(x)]\zeta^\gamma
 dS(\zeta)\right]\xi^\gamma dS(\xi)=
$$
$$
=\frac{1}{|S_1^+(n)|_\gamma^2}\int\limits_{S_1^+(n)}T_x^{\mu\xi}\left[\lim\limits_{R\rightarrow +\infty}\int\limits_{B_R^+(n)}T_x^{z}[f(x)]g(|z|)z^\gamma
 dz\right]\xi^\gamma dS(\xi)=
$$
$$
=\frac{1}{|S_1^+(n)|_\gamma^2}\int\limits_{S_1^+(n)}T_x^{\mu\xi}\left[\int\limits_{\mathbb{R}^+_n}T_x^{z}[f(x)]g(|z|)z^\gamma
 dz\right]\xi^\gamma dS(\xi).
$$

Теперь, применяя свойства ассоциативности и самосопряжённости обобщённого сдвига (см. \cite{levitan}, формулы (7.3) и (7.4), соответственно), получим
$$
J=\frac{1}{|S_1^+(n)|_\gamma^2}\int\limits_{S_1^+(n)}\int\limits_{\mathbb{R}_n^+}\left( T_z^{\mu\xi}T_z^x[f(z)]\right)\,g(|z|)z^\gamma
dz\,\,\xi^\gamma dS(\xi)=
$$
$$
=\frac{1}{|S_1^+(n)|_\gamma^2}\int\limits_{S_1^+(n)}\int\limits_{\mathbb{R}_n^+}T_z^{x}[f(z)]\,\,T_z^{\mu\xi}
[g(|z|)]\,z^\gamma dz\,\,\xi^\gamma dS(\xi)=
$$
$$
=\frac{C(\gamma)}{|S_1^+(n)|_\gamma^2}\int\limits_{S_1^+(n)}\int\limits_{\mathbb{R}_n^+}T_x^{z}[f(x)]\int\limits_0^\pi...\int\limits_0^\pi
\prod\limits_{i=1}^n\sin^{\gamma_i-1}\alpha_i\times
$$
$$
\times
g(\sqrt{\mu^2\xi_1^2+...+\mu^2\xi_n^2+z_1^2+...+z_n^2-2\mu\xi_1z_1\cos\alpha_1-...-2\mu\xi_nz_n\cos\alpha_n})\times
$$
$$
\times d\alpha_1...d\alpha_n\, z^\gamma dz\,\,\xi^\gamma dS(\xi).
$$
В интеграле по $z$ перейдем к сферическим координатам $z=r\eta,$ $|\eta|=1$, $r\geq 0$. Учитывая, что $|\xi|=1$, получим
$$
J=\frac{C(\gamma)}{|S_1^+(n)|_\gamma^2}\int\limits_{S_1^+(n)}\int\limits_{0}^{+\infty}r^{n+|\gamma|-1}
dr\int\limits_{S_1^+(n)}T^{r\eta}_x[f(x)]\times
$$
$$\times\int\limits_0^\pi...\int\limits_0^\pi
\prod\limits_{i=1}^n\sin^{\gamma_i-1}\alpha_i
g(\sqrt{\mu^2+r^2-2r\mu\langle \xi,\eta\cos\alpha\rangle })d\alpha_1...d\alpha_n\,
\eta^\gamma dS(\eta)\,\,\xi^\gamma dS(\xi),
$$
где
$$
\langle \xi,\eta\cos\alpha\rangle=\xi_1\eta_1\cos\alpha_1+...+\xi_n\eta_n\cos\alpha_n.
$$
Используя многомерный оператор Пуассона \eqref{Poison}, интеграл $J$ запишем в виде
$$
J=\frac{1}{|S_1^+(n)|_\gamma^2}\int\limits_{0}^{+\infty}r^{n+|\gamma|-1}dr\int\limits_{S_1^+(n)}T^{r\eta}_{x}[f(x)]\eta^\gamma dS(\eta)\times$$
$$\times
\int\limits_{S_1^+(n)}\mathcal{P}_\eta^\gamma g(\sqrt{r^2+\mu^2-2r\mu\langle
\xi,\eta\rangle })\,\,\xi^\gamma
dS(\xi).
$$
Применяя к интегралу
$$\int\limits_{S_1^+(n)}\mathcal{P}_\eta^\gamma g(\sqrt{r^2+\mu^2-2r\mu\langle
\xi,\eta\rangle })\,\,\xi^\gamma
dS(\xi)
$$
формулу \eqref{VesPlVol} и учитывая, что $|\eta|=1$,  получим
$$
J=\frac{\prod\limits^n_{i=1}\Gamma\left(\frac{\gamma_i+1}{2}\right)}{\sqrt{\pi}2^{n-1}\Gamma\left(\frac{|\gamma|+n-1}{2}\right)}\frac{1}{|S_1^+(n)|_\gamma^2}
\int\limits_{0}^{+\infty}r^{n+|\gamma|-1}dr \int\limits_{S_1^+(n)} T^{r\eta}_x[f(x)]\eta^\gamma
dS(\eta)\times
$$
$$
\times\int\limits_{-1}^1
    (1-p^2)^{\frac{n+|\gamma|-3}{2}}g(\sqrt{r^2+\mu^2-2r\mu p})\,\,dp=
$$
$$
=\frac{\prod\limits^n_{i=1}\Gamma\left(\frac{\gamma_i+1}{2}\right)}{\sqrt{\pi}2^{n-1}\Gamma\left(\frac{|\gamma|+n-1}{2}\right)}\frac{1}{|S_1^+(n)|_\gamma}
\int\limits_{0}^{+\infty}r^{n+|\gamma|-1}\,M^\gamma_f(x;r)\,dr\times
$$
$$
\times\int\limits_{-1}^1
    (1-p^2)^{\frac{n+|\gamma|-3}{2}}g(\sqrt{r^2+\mu^2-2r\mu p})\,\,dp.
$$
Теперь вместо переменной $p$ введем переменную $\lambda$, связанную с $p$ формулой
$$r^2+\mu^2-2r\mu p=\lambda^2.$$
Вычисляем значения
$$p=\frac{r^2+\mu^2-\lambda^2}{2r\mu},
dp=-\frac{\lambda}{r\mu}\ d\lambda,\ \
1-p^2=\frac{(\lambda^2-(r-\mu)^2)((r+\mu)^2-\lambda^2)}{(2r\mu)^2},
$$
и при $p=-1$ получаем $\lambda=|r+\mu|$, а при $p=1$ получаем $\lambda=|r-\mu|$. Тогда,  используя \eqref{PlSh}, выводим
$$
J=(2\mu)^{{2-n-|\gamma|}}\frac{2\,\Gamma\left(\frac{n{+}|\gamma|}{2}\right)}{\sqrt{\pi}\Gamma\left(\frac{|\gamma|+n-1}{2}\right)}
\times
$$
$$\times\int\limits_{0}^{+\infty}g(\lambda)\,\lambda\, d\lambda\int\limits_{|\lambda-\mu|}^{|\lambda+\mu|}
    \left({(\lambda^2-(r-\mu)^2)((r+\mu)^2-\lambda^2)}\right)^{\frac{n+|\gamma|-3}{2}}\,M^\gamma_f(x;r)\,rdr.
$$
Поскольку $g(\lambda)$ --- произвольная функция, то из того что
$$
\int\limits_0^{+\infty}\lambda^{n+|\gamma|-1}
g(\lambda)I^{\,\gamma}_f(x;\lambda,\mu)d\lambda=
(2\mu)^{{2-n-|\gamma|}}\frac{2\,\Gamma\left(\frac{n{+}|\gamma|}{2}\right)}{\sqrt{\pi}\Gamma\left(\frac{|\gamma|+n-1}{2}\right)}
\times
$$
$$\times\int\limits_{0}^{+\infty}g(\lambda)\,\lambda\, d\lambda\int\limits_{\lambda-\mu}^{\lambda+\mu}
    \Bigl({(\lambda^2-(r-\mu)^2)((r+\mu)^2-\lambda^2)}\Bigr)^{\frac{n+|\gamma|-3}{2}}\,M^\gamma_f(x;r)\,rdr
$$
следует, что
$$
I^{\,\gamma}_f(x;\lambda,\mu)=\frac{2\,\Gamma\left(\frac{n{+}|\gamma|}{2}\right)}{\sqrt{\pi}\Gamma\left(\frac{|\gamma|+n-1}{2}\right)}\times
$$
\begin{equation}\label{04}
\times\frac{1}{(2\lambda\mu)^{{n+|\gamma|-2}}}\int\limits_{\lambda-\mu}^{\lambda+\mu}
    \Bigl({(\lambda^2-(r-\mu)^2)((r+\mu)^2-\lambda^2)}\Bigr)^{\frac{n+|\gamma|-3}{2}}\,M^\gamma_f(x;r)\,rdr.
\end{equation}
Мы убрали модули в пределах интегрирования в силу того, что подынтегральная функция нечетна по $r$.

Теперь, используя представление \eqref{sdvig2} обобщённого сдвига  при $\nu{=}n{+}|\gamma|{-}1$, получим
$$
I^{\,\gamma}_f(x;\lambda,\mu)=\,^\nu\mathcal{T}_\mu^t M^\gamma_f(x;\mu).
$$

Если в формуле \eqref{04} мы положим $\alpha=\lambda-\mu$, $\beta=\lambda+\mu$, ($\beta>\alpha$), то можем записать
$$
I^{\,\gamma}_f\left(x;\frac{\beta-\alpha}{2},\frac{\beta+\alpha}{2}\right)=
$$
$$
=\frac{\Gamma\left(\frac{n{+}|\gamma|}{2}\right)}{\sqrt{\pi}\Gamma\left(\frac{|\gamma|+n-1}{2}\right)}
\frac{2^{{n+|\gamma|-1}}}{(\beta^2-\alpha^2)^{{n+|\gamma|-2}}}\int\limits_{\alpha}^{\beta}
    \Bigl({(\beta^2-r^2)(r^2-\alpha^2)}\Bigr)^{\frac{n+|\gamma|-3}{2}}M^\gamma_f(x;r)rdr.
$$
Теорема доказана.
\end{proof}

Из доказанной теоремы вытекают формулы, выражающие действие итерированных весовых средних на функции Бесселя.
Для этого в теореме 1 выберем $f(x)=\mathbf{j}_\gamma(x,\xi)$.
Далее, в \cite{LPSh1} показано, что
$$
M^\gamma_\mu{\bf j}_\gamma(x,\xi)
={\bf j}_\gamma(x,\xi)\,\,j_{\frac{n+|\gamma|-2}{2}}(\mu).
$$
Тогда
$$
I^{\,\gamma}_{\lambda,\mu}{\bf j}_\gamma(x,\xi)= {\bf j}_\gamma(x,\xi)\,\,j_{\frac{n+|\gamma|-2}{2}}(\mu)\,\,j_{\frac{n+|\gamma|-2}{2}}(\lambda).
$$

Если для $I^{\,\gamma}_{\lambda,\mu}{\bf j}_\gamma(x,\xi)$ записать равенство \eqref{Transl}, то получим известную формулу
$$
\,^\nu\mathcal{T}_\mu^\lambda j_{\frac{\nu-1}{2}}(\mu)= j_{\frac{\nu-1}{2}}(\mu)\,\,j_{\frac{\nu-1}{2}}(\lambda),\qquad \nu=n+|\gamma|-1,
$$
полученную в \cite{levitan} (см. стр. 124), где она выводится из анализа дифференциального уравнения, которому удовлетворяет обобщённый сдвиг. Мы же получили прямое доказательство этой формулы, выведя её непосредственно из теоремы 1 с помощью весовых сферических средних.

\section{Применения тождества для итерированного сферического среднего к задаче компьютерной томографии}

Рассмотрим  применение формулы \eqref{IterFor} из полученной выше теоремы 1 к компьютерной томографии. В задачах  дифракционной томографии и обратного рассеяния преобразование Фурье--Бесселя (Ханкеля) функции представляет собой измеренные данные (см., например, \cite{Higgins},  \cite{Kuchment} p. 126, \cite{Halliwell} p. 90). Подобные формулы используются для восстановления функции по её известным сферическим средним.

Докажем формулу, выражающую функцию через ее преобразование Фурье--Бесселя и обобщённый сдвиг. Это обобщение известной формулы для более простой задачи, в которой используется представление функции через её преобразование Фурье и обычный сдвиг \cite{Beylkin}. В такой форме подобные представления используются для восстановления функции в указанных задачах томографии и интегральной геометрии.

Отметим, что в  теории рассеяния поверхность шара $|x|<2\lambda$, которая далее используется при интегрировании,  называется сферой Эвальда (см. \cite{Ewald}).

\begin{theorem} Пусть преобразование Фурье--Бесселя $\widehat{F}$ --- функция с носителем внутри части шара
$$B_{2\lambda}(n)^+=\{x\in\mathbb{R}_n^+:|x|<2\lambda\}.$$
Тогда  справедливо равенство
\begin{equation}\label{teo2}
F(y)=C(n,\gamma)\int\limits_{S^+_1(n)}\int\limits_{S^+_1(n)}T_{\lambda\xi}^{\lambda\zeta}\left[\frac{|\lambda\xi|\widehat{F}(\lambda\xi){\bf j}_\gamma(\lambda\xi,y)}{(4\lambda^2-|t\xi|^2)^{\frac{n+|\gamma|-3}{2}}}\right]\zeta^\gamma
\xi^\gamma dS(\xi) dS(\zeta),
\end{equation}
$$
C(n,\gamma)=\frac{\sqrt{\pi}2^{2n-3}\lambda^{{2n+2|\gamma|-4}}\Gamma\left(\frac{|\gamma|+n-1}{2}\right)}{\Gamma^2\left(\frac{n{+}|\gamma|}{2}\right)\prod\limits_{j=1}^n\Gamma
\left(\frac{\gamma_j+1}{2}\right)|S_1^+(n)|_\gamma^2}.
$$
\end{theorem}
\begin{proof}
В  \eqref{IterFor} положим $\mu=\lambda$, получим
$$
I^{\,\gamma}_f(x;\lambda,\lambda)=
$$
\begin{equation}\label{Itermeans1}
=\frac{2\,\Gamma\left(\frac{n{+}|\gamma|}{2}\right)}{\sqrt{\pi}\Gamma\left(\frac{|\gamma|+n-1}{2}\right)}\frac{1}{(2\lambda^2)^{{n+|\gamma|-2}}}\int\limits_{0}^{2\lambda}
    \left[4\lambda^2{-}r^2\right]^{\frac{n{+}|\gamma|{-}3}{2}}\,r^{n{+}|\gamma|{-}2}\,M^\gamma_f(x;r)dr.
\end{equation}

Рассмотрим функцию
$$
f_y(x)=\frac{|x|}{(4\lambda^2-|x|^2)^{\frac{n+|\gamma|-3}{2}}}\widehat{F}(x){\bf j}_\gamma(x,y),\qquad x,y\in\mathbb{R}_n^+.
$$
Найдем её весовое сферическое среднее
$f_y(x)$ при $x=0$:
$$
M^\gamma_r f_y(x)|_{x=0}=\frac{1}{|S_1^+(n)|_\gamma}\int\limits_{S_1^+(n)}[T_x^{rz}f_y(x)]_{x=0}z^\gamma
dS(z)=
$$
$$
=\frac{1}{|S_1^+(n)|_\gamma}\int\limits_{S_1^+(n)}\left[T_x^{rz}\frac{|x|}{(4\lambda^2-|x|^2)^{\frac{n+|\gamma|-3}{2}}}\widehat{F}(x){\bf j}_\gamma(x,y)\right]_{x=0}z^\gamma
dS(z)=
$$
$$
=\frac{1}{|S_1^+(n)|_\gamma}\int\limits_{S_1^+(n)}\frac{r|z|}{(4\lambda^2-r|z|^2)^{\frac{n+|\gamma|-3}{2}}}\widehat{F}(rz){\bf j}_\gamma(rz,y)z^\gamma
dS(z)=
$$
$$
=\frac{1}{|S_1^+(n)|_\gamma}\frac{r}{(4\lambda^2-|r|^2)^{\frac{n+|\gamma|-3}{2}}}\int\limits_{S_1^+(n)}\widehat{F}(rz){\bf j}_\gamma(rz,y)z^\gamma
dS(z).
$$
Теперь, используя \eqref{Itermeans1}, найдем
$$
I_f^\gamma(0;\lambda,\lambda)=
$$
$$
=\frac{2\,\Gamma\left(\frac{n{+}|\gamma|}{2}\right)}
{\sqrt{\pi}\Gamma\left(\frac{|\gamma|+n-1}{2}\right)}
\frac{1}{(2\lambda^2)^{{n+|\gamma|-2}}}\,\int\limits_{0}^{2\lambda}
    \left[4\lambda^2-r^2\right]^{\frac{n+|\gamma|-3}{2}}\,r^{n+|\gamma|-2}\,M_f^\gamma(0,r)dr{=}
$$
$$
{=}\frac{2\,\Gamma\left(\frac{n{+}|\gamma|}{2}\right)}
{\sqrt{\pi}\Gamma\left(\frac{|\gamma|+n-1}{2}\right)}\,
\frac{(2\lambda^2)^{2-n-|\gamma|}}{|S_1^+(n)|_\gamma}\int\limits_{0}^{2\lambda}
    \left[4\lambda^2-r^2\right]^{\frac{n+|\gamma|-3}{2}}
    \,r^{n+|\gamma|-2}\frac{r}{(4\lambda^2-r^2)^{\frac{n+|\gamma|-3}{2}}}\,dr\times
$$
$$   \times\int\limits_{S_1^+(n)}\widehat{F}(rz){\bf j}_\gamma(rz,y)z^\gamma
dS(z)=
$$
$$
=\frac{2\,\Gamma\left(\frac{n{+}|\gamma|}{2}\right)}{\sqrt{\pi}\Gamma
\left(\frac{|\gamma|+n-1}{2}\right)}\,\frac{(2\lambda^2)^{2-n-|\gamma|}}
{|S_1^+(n)|_\gamma}\int\limits_{0}^{2\lambda}
    \,r^{n+|\gamma|-1}\,dr\int\limits_{S_1^+(n)}\widehat{F}(rz){\bf j}_\gamma(rz,y)z^\gamma
dS(z).
$$
Применяя формулу \eqref{Ball}, получим

$$
I_f^\gamma(0;\lambda,\lambda)=\frac{2\,\Gamma\left(\frac{n{+}|\gamma|}{2}\right)}{\sqrt{\pi}\Gamma\left(\frac{|\gamma|+n-1}{2}\right)}\,\frac{(2\lambda^2)^{2-n-|\gamma|}}{|S_1^+(n)|_\gamma}\int\limits_{B_{2\lambda}^+}\widehat{F}(z){\bf j}_\gamma(z,y)z^\gamma
dz=
$$
$$
=\frac{2\,\Gamma\left(\frac{n{+}|\gamma|}{2}\right)}{\sqrt{\pi}\Gamma\left(\frac{|\gamma|+n-1}{2}\right)}\,\frac{(2\lambda^2)^{2-n-|\gamma|}}{|S_1^+(n)|_\gamma}\int\limits_{\mathbb{R}^+_n}\widehat{F}(z){\bf j}_\gamma(z,y)z^\gamma
dz=
$$
$$
=\frac{2\,\Gamma\left(\frac{n{+}|\gamma|}{2}\right)}{\sqrt{\pi}\Gamma\left(\frac{|\gamma|+n-1}{2}\right)}\,\frac{(2\lambda^2)^{2-n-|\gamma|}}{|S_1^+(n)|_\gamma}\,2^{|\gamma|-n}\prod\limits_{j=1}^n\Gamma^2
\left(\frac{\gamma_j+1}{2}\right)\, F(y)=
$$
$$
=\frac{\Gamma^2\left(\frac{n{+}|\gamma|}{2}\right)\prod\limits_{j=1}^n\Gamma
\left(\frac{\gamma_j+1}{2}\right)}{\sqrt{\pi}2^{2n-3}\lambda^{{2n+2|\gamma|-4}}\Gamma\left(\frac{|\gamma|+n-1}{2}\right)}\, F(y).
$$
Следовательно,
\begin{equation}\label{ItF}
   F(y)= \frac{\sqrt{\pi}2^{2n-3}\lambda^{{2n+2|\gamma|-4}}\Gamma\left(\frac{|\gamma|+n-1}{2}\right)}{\Gamma^2\left(\frac{n{+}|\gamma|}{2}\right)\prod\limits_{j=1}^n\Gamma
\left(\frac{\gamma_j+1}{2}\right)}\,I_f^\gamma(0;\lambda,\lambda).
\end{equation}

С другой стороны
\begin{equation}\label{Itf2}
I^\gamma_f(0;\lambda,\lambda)=\frac{1}{|S_1^+(n)|_\gamma^2}\int\limits_{S_1^+(n)}\int\limits_{S_1^+(n)}T_{t\xi}^{t\zeta}[f(t\xi)]\zeta^\gamma
\xi^\gamma dS(\xi) dS(\zeta).
\end{equation}

Из \eqref{ItF} и \eqref{Itf2} получим
$$
F(y)=\frac{\sqrt{\pi}2^{2n-3}\lambda^{{2n+2|\gamma|-4}}\Gamma\left(\frac{|\gamma|+n-1}{2}\right)}{\Gamma^2\left(\frac{n{+}|\gamma|}{2}\right)\prod\limits_{j=1}^n\Gamma
\left(\frac{\gamma_j+1}{2}\right)|S_1^+(n)|_\gamma^2}\times$$
$$\times\int\limits_{S^+_1(n)}\int\limits_{S^+_1(n)}T_{\lambda\xi}^{\lambda\zeta}\left[\frac{|\lambda\xi|}{(4\lambda^2-|t\xi|^2)^{\frac{n+|\gamma|-3}{2}}}\widehat{F}(\lambda\xi){\bf j}_\gamma(\lambda\xi,y)\right]\zeta^\gamma
\xi^\gamma dS(\xi) dS(\zeta).
$$
\end{proof}

\end{document}